\documentclass[11pt,leqno]{article}
\usepackage[margin=1in]{geometry} 
\usepackage{amssymb,amsfonts,amsmath,bbm,mathrsfs,stmaryrd,mathtools,mathabx}
\usepackage{xcolor}
\usepackage{url}

\usepackage{extarrows}

\usepackage[shortlabels]{enumitem}
\usepackage{tensor}

\usepackage{xr}
\usepackage[T1]{fontenc}

\usepackage[colorlinks,
linkcolor=black!75!red,
citecolor=blue,
pdftitle={},
pdfproducer={pdfLaTeX},
pdfpagemode=None,
bookmarksopen=true,
bookmarksnumbered=true,
backref=page]{hyperref}

\usepackage{tikz}
\usetikzlibrary{arrows,calc,decorations.pathreplacing,decorations.markings,intersections,shapes.geometric,through,fit,shapes.symbols,positioning,decorations.pathmorphing,shapes.misc}

\usepackage{braket}

\usepackage[amsmath,thmmarks,hyperref]{ntheorem}
\usepackage{cleveref}

\newcommand\nthalias[1]{\AddToHook{env/#1/begin}{\crefalias{lemma}{#1}}}

\nthalias{definition}
\nthalias{example}
\nthalias{examples}
\nthalias{remark}
\nthalias{remarks}
\nthalias{convention}
\nthalias{notation}
\nthalias{construction}
\nthalias{theoremN}
\nthalias{propositionN}
\nthalias{corollaryN}
\nthalias{lemma}
\nthalias{proposition}
\nthalias{corollary}
\nthalias{theorem}
\nthalias{conjecture}
\nthalias{question}
\nthalias{assumption}

\creflabelformat{enumi}{#2#1#3}

\crefname{section}{Section}{Sections}
\crefformat{section}{#2Section~#1#3} 
\Crefformat{section}{#2Section~#1#3} 

\crefname{subsection}{\S}{\S\S}
\AtBeginDocument{%
  \crefformat{subsection}{#2\S#1#3}%
  \Crefformat{subsection}{#2\S#1#3}% 
}

\crefname{subsubsection}{\S}{\S\S}
\AtBeginDocument{%
  \crefformat{subsubsection}{#2\S#1#3}%
  \Crefformat{subsubsection}{#2\S#1#3}% 
}

\theoremstyle{plain}

\newtheorem{lemma}{Lemma}[section]
\newtheorem{proposition}[lemma]{Proposition}

\newtheorem{theorem}[lemma]{Theorem}

\theoremstyle{plain}
\theoremnumbering{Alph}
\newtheorem{theoremN}{Theorem}

\theoremstyle{plain}
\theorembodyfont{\upshape}
\theoremsymbol{\ensuremath{\blacklozenge}}

\newtheorem{definition}[lemma]{Definition}

\newtheorem{remark}[lemma]{Remark}
\newtheorem{remarks}[lemma]{Remarks}

\crefname{definition}{definition}{definitions}
\crefformat{definition}{#2definition~#1#3} 
\Crefformat{definition}{#2Definition~#1#3} 

\crefname{ex}{example}{examples}
\crefformat{example}{#2example~#1#3} 
\Crefformat{example}{#2Example~#1#3} 

\crefname{exs}{example}{examples}
\crefformat{examples}{#2example~#1#3} 
\Crefformat{examples}{#2Example~#1#3} 

\crefname{remark}{remark}{remarks}
\crefformat{remark}{#2remark~#1#3} 
\Crefformat{remark}{#2Remark~#1#3} 

\crefname{remarks}{remark}{remarks}
\crefformat{remarks}{#2remark~#1#3} 
\Crefformat{remarks}{#2Remark~#1#3} 

\crefname{convention}{convention}{conventions}
\crefformat{convention}{#2convention~#1#3} 
\Crefformat{convention}{#2Convention~#1#3} 

\crefname{notation}{notation}{notations}
\crefformat{notation}{#2notation~#1#3} 
\Crefformat{notation}{#2Notation~#1#3} 

\crefname{table}{table}{tables}
\crefformat{table}{#2table~#1#3} 
\Crefformat{table}{#2Table~#1#3} 

\crefname{lemma}{lemma}{lemmas}
\crefformat{lemma}{#2lemma~#1#3} 
\Crefformat{lemma}{#2Lemma~#1#3} 

\crefname{proposition}{proposition}{propositions}
\crefformat{proposition}{#2proposition~#1#3} 
\Crefformat{proposition}{#2Proposition~#1#3} 

\crefname{propositionN}{proposition}{propositions}
\crefformat{propositionN}{#2proposition~#1#3} 
\Crefformat{propositionN}{#2Proposition~#1#3} 

\crefname{corollary}{corollary}{corollaries}
\crefformat{corollary}{#2corollary~#1#3} 
\Crefformat{corollary}{#2Corollary~#1#3} 

\crefname{corollaryN}{corollary}{corollaries}
\crefformat{corollaryN}{#2corollary~#1#3} 
\Crefformat{corollaryN}{#2Corollary~#1#3} 

\crefname{theorem}{theorem}{theorems}
\crefformat{theorem}{#2theorem~#1#3} 
\Crefformat{theorem}{#2Theorem~#1#3} 

\crefname{theoremN}{theorem}{theorems}
\crefformat{theoremN}{#2theorem~#1#3} 
\Crefformat{theoremN}{#2Theorem~#1#3} 

\crefname{enumi}{}{}
\crefformat{enumi}{#2#1#3}
\Crefformat{enumi}{#2#1#3}

\crefname{assumption}{assumption}{Assumptions}
\crefformat{assumption}{#2assumption~#1#3} 
\Crefformat{assumption}{#2Assumption~#1#3} 

\crefname{construction}{construction}{Constructions}
\crefformat{construction}{#2construction~#1#3} 
\Crefformat{construction}{#2Construction~#1#3} 

\crefname{question}{question}{Questions}
\crefformat{question}{#2question~#1#3} 
\Crefformat{question}{#2Question~#1#3} 

\crefname{equation}{}{}
\crefformat{equation}{(#2#1#3)} 
\Crefformat{equation}{(#2#1#3)}

\numberwithin{equation}{section}

\theoremstyle{nonumberplain}
\theoremsymbol{\ensuremath{\blacksquare}}

\newtheorem{proof}{Proof}
\newcommand\pf[1]{\newtheorem{#1}{Proof of \Cref{#1}}}

\newcommand\bR{{\mathbb R}}

\newcommand\bZ{{\mathbb Z}}

\newcommand\cA{{\mathcal A}}

\newcommand\cC{{\mathcal C}}

\newcommand\cE{{\mathcal E}}

\newcommand\cK{{\mathcal K}}

\newcommand\cS{{\mathcal S}}

\DeclareMathOperator{\id}{id}

\DeclareMathOperator{\im}{\mathrm{im}}

\newcommand{\cat}[1]{\textsc{#1}}

\newcommand{\qedhere}{\mbox{}\hfill\ensuremath{\blacksquare}}

\newcommand{\xrightarrowdbl}[2][]{%
  \xrightarrow[#1]{#2}\mathrel{\mkern-14mu}\rightarrow
}

\title{Single and multi-valued Hilbert-bundle renormings}
\author{Alexandru Chirvasitu}

\begin{document}

\date{}

\newcommand{\Addresses}{{% additional braces for segregating \footnotesize
  \bigskip
  \footnotesize

  \textsc{Department of Mathematics, University at Buffalo}
  \par\nopagebreak
  \textsc{Buffalo, NY 14260-2900, USA}  
  \par\nopagebreak
  \textit{E-mail address}: \texttt{achirvas@buffalo.edu}

  % % \medskip
  % % 
  % % \textsc{Department of Mathematics, INSTITUTION}
  % % \par\nopagebreak
  % % \textsc{ADDRESS}
  % % \par\nopagebreak
  % % \textit{E-mail address}: \texttt{??}
  % % 

}}

\maketitle

\begin{abstract}
  We prove that subhomogeneous continuous Banach bundles over compact metrizable spaces are equivalent to Hilbert bundles, while examples show that the metrizability assumption cannot be dropped completely. This complements the parallel statement for homogeneous bundles without the metrizability assumption, and generalizes the analogous result to the effect that subhomogeneous $C^*$ bundles over compact metrizable spaces admit finite-index expectations. 
\end{abstract}

\noindent {\em Key words:
  Banach bundle;
  Hilbert bundle;
  convex structure;
  lower semicontinuous;
  renorming;
  selection theorem;
  seminorm;
  subhomogeneous
  
}

\vspace{.5cm}

\noindent{MSC 2020:
  46B03; %Isomorphic theory (including renorming) of Banach spaces
  46M20; %Methods of algebraic topology in functional analysis (cohomology, sheaf and bundle theory, etc.)
  54C65; %Selections in general topology
  46C05; %Hilbert and pre-Hilbert spaces: geometry and topology (including spaces with semidefinite inner product)
  52A21; %Convexity and finite-dimensional Banach spaces (including special norms, zonoids, etc.) (aspects of convex geometry)  
  52A01; %Axiomatic and generalized convexity
  54E35; %Metric spaces, metrizability
  54B20 %Hyperspaces in general topology
  
}

%\tableofcontents

%%%%%%%%%%%%%%%%%%%%%%%%%%%%%%%%%%%%%%%%%%%%%%%%%%%%%%%%%%%%%%%%%%%%%%%%%%%%%
%%%%%%%%%%%%%%%%%%%%%%%%%%%%%%%%%%%%%%%%%%%%%%%%%%%%%%%%%%%%%%%%%%%%%%%%%%%%%
\section*{Introduction}

The present paper's Banach bundles (occasionally \emph{continuous} Banach bundles for emphasis) are those of \cite[\S 13.4]{fd_bdl-1}: continuous open surjections $\cE\xrightarrowdbl{\pi}X$ such that
\begin{itemize}[wide]
\item the \emph{fibers} $\cE_x:=\pi^{-1}(x)$, $x\in X$ come equipped with Banach norms $\|\cdot\|_x$ gluing to a continuous map $\cE\xrightarrow{\|\cdot\|}\bR_{\ge 0}$;

\item fiber-wise scalar multiplication and addition is continuous;

\item a \emph{net}'s \cite[\S 1.6]{eng_top_1989} convergence to the $x$-fiber's zero element $0_x\in \cE_x$ is characterized by
  \begin{equation*}
    p_{\lambda}
    \xrightarrow[\quad\lambda\quad]{}
    0_x
    \iff
    \left(
      \left\|p_{\lambda}\right\|
      \xrightarrow[\quad\lambda\quad]{}
      0
    \right)
    \&
    \left(
      \pi\left(p_{\lambda}\right)
      \xrightarrow[\quad\lambda\quad]{}
      x
    \right).
  \end{equation*}
\end{itemize}
In particular, the somewhat broader concept of \cite[pp.7-8]{dg_ban-bdl} whereby $\|\cdot\|$ is only \emph{upper semicontinuous} \cite[Problem 1.7.14]{eng_top_1989} will not make much of an appearance. 

The motivating impetus for the ensuing discussion is provided by a cluster of problems touched upon explicitly in \cite{MR3056657,bg_cx-exp} and somewhat obliquely in \cite{pvl_disc} revolving around the possibility of recasting a given Banach bundle as a \emph{Hilbert bundle} in the sense of \cite[Definition 13.5]{fd_bdl-1}: one whose fibers are Hilbert spaces under their norms $\|\cdot\|_x$. We elaborate below. 

A \emph{renorming} of a Banach bundle $\cE\xrightarrowdbl{}X$ (sometimes \emph{equivalent renorming} for emphasis) is a Banach-bundle norm $\vvvert\cdot\vvvert$ equivalent to the original $\|\cdot\|$ in the sense that
\begin{equation*}
  \left(\Gamma_b(\cE),\ \|\cdot\|\right)
  \xrightarrow[\quad\cong\quad]{\quad\id=\text{top-linear isomorphism}\quad}
  \left(\Gamma_b(\cE),\ \vvvert\cdot\vvvert\right),
\end{equation*}
where for any Banach bundle we write
\begin{equation*}
  \begin{aligned}
    \Gamma(\cE)
    &:=
      \left\{X\xrightarrow[\text{continuous}]{s}\cE\ :\ \pi s=\id_X\right\}
      \quad
      \left(\text{\emph{sections} of $\cE$ \cite[p.9]{dg_ban-bdl}}\right)
      \quad\text{and}\\
    \Gamma_b(\cE)
    &:=
      \left\{s\in \Gamma(\cE)\ :\ \|s\|:=\sup_x \|s(x)\|_x<\infty\right\}
      \quad
      \left(\text{bounded sections}\right).
  \end{aligned}  
\end{equation*}
As expected, \emph{Hilbert} renormings will be those for which $\left(\cE_x,\vvvert\cdot\vvvert_x\right)$ Hilbert spaces. Recall (e.g. from the discussion surrounding \cite[Proposition 2.1]{gog_top-fg}) that a Banach bundle is
\begin{itemize}[wide]
\item \emph{homogeneous} if all of its fibers $\cE_x$ have the same finite dimension;

\item and \emph{subhomogeneous} if there is a finite upper bound on its fiber dimensions (a term borrowed from the $C^*$-algebra literature \cite[Definition IV.1.4.1]{blk}):
  \begin{equation*}
    \sup_{x\in X}\dim \cE_x<\infty.
  \end{equation*}
\end{itemize}

Having settled the requisite language, some of literature background is as follows. 
\begin{itemize}[wide]
\item A slightly re-tailored \cite[Problem 4.7]{MR3056657} asks whether homogeneous Banach bundles over compact Hausdorff spaces admit Hilbert renormings. \cite[Theorem B]{2409.03531v1} answers this affirmatively with a sharp bound for ``how far apart'' the two norms (original and Hilbert) can be.

\item The former paper's \cite[Problem 4.9]{MR3056657} also asks a variant question: do subhomogeneous (unital) \emph{$C^*$ bundles} $\cA\xrightarrowdbl{\pi}X$ (i.e. \cite[p.9]{dg_ban-bdl} those whose fibers are unital $C^*$-algebras with the relevant operations continuous) over compact Hausdorff spaces admit Hilbert renormings?

  The original formulation rather asks whether there is an \emph{expectation} (i.e. \cite[Theorem II.6.10.1 and Theorem II.6.10.2]{blk} a norm-1 projection) $\Gamma(\cA)\xrightarrowdbl{E}C(X)$ of \emph{finite index} in the sense of \cite[Definition 2]{fk_fin-ind}:
  \begin{equation*}
    \left(\exists K\in \bR_{\ge 1}\right)
    \left(\forall s\in \Gamma(\cA)\right)
    \quad:\quad
    KE(s^*s)\ge s^*s.
  \end{equation*}
  That the two versions of the question are in fact equivalent follows from \cite[Proposition 5.4]{pt_brnch}.

  \cite[Theorem A]{2409.03531v1} answers the question affirmatively for \emph{metrizable} compact $X$ (recovering in particular \cite[Theorem 1.1]{pt_brnch} over metrizable base spaces), while \cite[Theorem A]{2409.17807v1} shows that in full generality the answer is negative, even if all $C^*$-algebras in sight are commutative.  

\item On the other hand, \cite[Proposition 3.4]{bg_cx-exp} proves that if the $C^*$ bundle $\cA\xrightarrowdbl{\pi}X$ is homogeneous (rather than only subhomogeneous) then Hilbert renormings do exist regardless of (the compact Hausdorff) base space. 
\end{itemize}

The following schematic summary aggregates the above items together with the obvious (in context) missing piece:

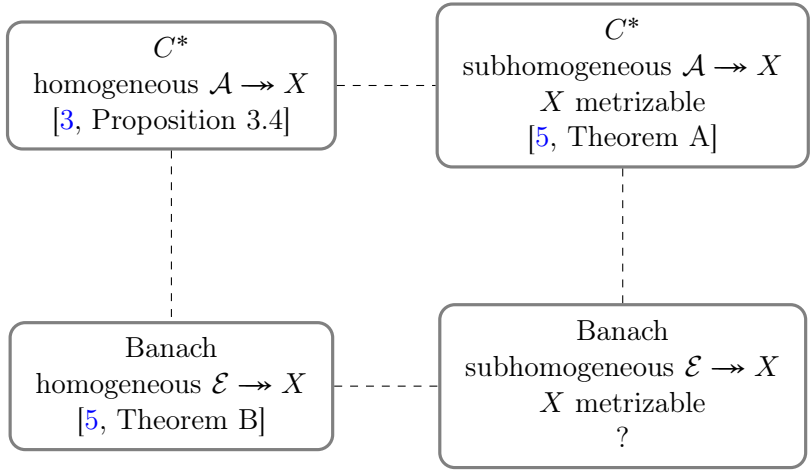
\begin{equation*}
  \begin{tikzpicture}[>=stealth,auto,baseline=(current  bounding  box.center)]
    \path[anchor=base]

    (0,0) node (cast_hom) [rectangle,minimum size=6mm,rounded corners=2mm,very thick,draw=black!50] {\begin{tabular}{c} $C^*$ \\ homogeneous $\cA\xrightarrowdbl{}X$\\\cite[Proposition 3.4]{bg_cx-exp}\end{tabular}}

    +(6,0) node (cast_subhom) [rectangle,minimum size=6mm,rounded corners=2mm,very thick,draw=black!50] {\begin{tabular}{c}$C^*$\\subhomogeneous $\cA\xrightarrowdbl{}X$\\ $X$ metrizable\\\cite[Theorem A]{2409.03531v1}\end{tabular}}

    +(0,-4) node (ban_hom) [rectangle,minimum size=6mm,rounded corners=2mm,very thick,draw=black!50] {\begin{tabular}{c}Banach \\ homogeneous $\cE\xrightarrowdbl{}X$\\\cite[Theorem B]{2409.03531v1}\end{tabular}}

    +(6,-4) node (ban_subhom) [rectangle,minimum size=6mm,rounded corners=2mm,very thick,draw=black!50] {\begin{tabular}{c}Banach \\ subhomogeneous $\cE\xrightarrowdbl{}X$\\$X$ metrizable\\?\end{tabular}}
    ;

    \draw[-,dashed] (cast_hom) to[bend left=0] node[pos=.5,auto] {$\scriptstyle $} (cast_subhom);
    \draw[-,dashed] (ban_hom) to[bend left=0] node[pos=.5,auto] {$\scriptstyle $} (ban_subhom);
    \draw[-,dashed] (cast_hom) to[bend left=0] node[pos=.5,auto] {$\scriptstyle $} (ban_hom);
    \draw[-,dashed] (cast_subhom) to[bend left=0] node[pos=.5,auto] {$\scriptstyle $} (ban_subhom);
  \end{tikzpicture}
\end{equation*}

One of the main goals of this note is to fill in the lower right-hand corner with an affirmative answer to the displayed question mark (\Cref{th:subhom.unitariz} below). The argument is very much in line with that proving \cite[Theorem A]{2409.03531v1}, ultimately relying on a version \cite[Theorem 3.4]{horv_top-cvx} of the celebrated \emph{Michael selection theorem} (\cite[Theorem 3.2'']{mich_contsel-1}, \cite[Theorem 1.1]{zbMATH06329568}, etc.).

It is precisely this selection aspect of the problem that demands metrizability, much as Michael-selection-type results typically do: see e.g. \cite[\S 5.1]{rs_cont-sel}, referencing \cite{zbMATH03569401,MR450972}, and \Cref{se:nbrnch} below for elaboration on the matter. It is perhaps worthwhile disentangling the portion of the argument relying on metrizability from the section that does not: \Cref{th:subhom.lsc.unitariz} shows that rather broadly, over \emph{locally paracompact} \cite[Defnition 3.1]{gierz_bdls} base spaces, one can effect a kind of \emph{multi-valued} renorming. 

\begin{theoremN}\label{thn:mv.renorm}
  Let $\cE\xrightarrowdbl{\pi}X$ be a subhomogeneous continuous Banach bundle over a locally paracompact space.

  There is a map
  \begin{equation*}
    X\ni x
    \xmapsto{\quad}
    \cK_x\subset \left\{\text{Hilbert norms on }\cE_x\right\}
  \end{equation*}
  with the following properties:
  \begin{itemize}[wide]
  \item Each $\cK_x$ is convex for an abstract \emph{convexity} \cite[Definition 2.1]{horv_top-cvx} on the space of seminorms on $\Gamma(\cE)$.

  \item Each $\cK_x$ is closed in the pointwise topology on seminorms.

  \item $\cK_{\bullet}$ is \emph{lower semicontinuous} \cite[\S 2]{mich_contsel-1} in $x\in X$ with respect to said pointwise topology.

  \item $\cK_{\bullet}$ is an \emph{equivalent} renorming in the sense of \Cref{def:mv.renorm}\Cref{item:def:mv.renorm:mv}:
    \begin{equation*}
      \sup_x
      \sup_{\vvvert\cdot\vvvert\in \cK(x)}
      d_{BM}\left((\cE_x,\|\cdot\|_x),\ (\cE_x,\vvvert\cdot\vvvert_x)\right)
      <\infty,
    \end{equation*}
    where $d_{BM}$ stands for \emph{Banach-Mazur distances} \cite[\S 37]{tj_bm}.  \qedhere
  \end{itemize}
\end{theoremN}

% % %%%%%%%%%%%%%%%%%%%%%%%%%%%%%%%%%%%%%%%%%%%%%%%%%%%%%%%%%%%%%%%%%%%%%%%%%%%%%
% % \subsection*{Acknowledgements}
% % 

% % %%%%%%%%%%%%%%%%%%%%%%%%%%%%%%%%%%%%%%%%%%%%%%%%%%%%%%%%%%%%%%%%%%%%%%%%%%%%%
% % %%%%%%%%%%%%%%%%%%%%%%%%%%%%%%%%%%%%%%%%%%%%%%%%%%%%%%%%%%%%%%%%%%%%%%%%%%%%%
% % \section{Preliminaries}\label{se:prel}
% %

%%%%%%%%%%%%%%%%%%%%%%%%%%%%%%%%%%%%%%%%%%%%%%%%%%%%%%%%%%%%%%%%%%%%%%%%%%%%%
%%%%%%%%%%%%%%%%%%%%%%%%%%%%%%%%%%%%%%%%%%%%%%%%%%%%%%%%%%%%%%%%%%%%%%%%%%%%%
\section{Hilbertable subhomogeneous bundles}\label{se:hlbrtbl}

\Cref{th:subhom.unitariz} below parallels \cite[Theorem B]{2409.03531v1}: compared to the latter, it allows for \emph{sub}homogeneous Banach bundles on the one hand, while imposing metrizability on the base space $X$ on the other.

\begin{theorem}\label{th:subhom.unitariz}
  A subhomogeneous continuous Banach bundle over a compact metrizable space admits an equivalent Hilbert-bundle renorming. 
\end{theorem}

We need a modified version of \Cref{th:subhom.unitariz}, weakening both the hypothesis and the conclusion.  Recall \cite[\S 2]{mich_contsel-1} that a map $X\xrightarrow{\varphi} 2^Y$ for topological spaces $X$ and $Y$ is \emph{lower semicontinuous (LSC)} if
\begin{equation*}
  \left(\forall x\in X\right)
  \left(\forall\text{ open }V\subseteq Y,\ \varphi(x)\cap V\ne\emptyset\right)
  \left(\exists\text{ nbhd }U\ni x\right)
  \left(\forall x'\in V\right)
  \quad:\quad
  \varphi(x')\cap V\ne\emptyset. 
\end{equation*}
Or (\cite[Definition 7.1.1]{kt_corresp}, \cite[Proposition 2.1]{mich_contsel-1}): $\varphi$ is continuous for the \emph{lower Vietoris topology} (\cite[\S 1.2]{ct_vietoris}, \cite[Definition 1.3.2]{kt_corresp}) on the power-set of $Y$.

To bring the aforementioned \cite[Theorem 3.2'']{mich_contsel-1} into scope, we will need weaker notions of renorming that allow for some ``slack'' in selecting the norms $\vvvert\cdot\vvvert_x$. We introduce the various terms.

\begin{definition}\label{def:mv.renorm}
  Let $\cE\xrightarrowdbl{\pi}X$ be a Banach bundle. 
  \begin{enumerate}[(1),wide]
  \item\label{item:def:mv.renorm:mv} A \emph{multi-valued (equivalent) renorming} for $\cE$ is a map
    \begin{equation}\label{eq:mv.renorm}
      X\ni x
      \xmapsto{\quad\cK\quad}
      \left(\text{set of norms on }\cE_x\right)
    \end{equation}
    with 
    \begin{equation}\label{eq:renorm.mv}
      \sup_x
      \sup_{\vvvert\cdot\vvvert\in \cK(x)}
      d_{BM}\left((\cE_x,\|\cdot\|_x),\ (\cE_x,\vvvert\cdot\vvvert_x)\right)
      <\infty,
    \end{equation}
    where $d_{BM}$ denotes the \emph{Banach-Mazur distances} \cite[\S 37]{tj_bm}
    \begin{equation*}
      d_{BM}\left((E,\|\cdot\|_E),\ (F,\|\cdot\|_F)\right)
      :=
      \sup\left\{\|T\|\cdot \|T^{-1}\|\right\}
      \quad\text{over}\quad
      (E,\|\cdot\|_E)
      \xrightarrow[\text{linear isomorphism}]{\quad T\quad}
      (F,\|\cdot\|_F).
    \end{equation*}

  \item\label{item:def:mv.renorm:lsc} Such a renorming is \emph{lower semicontinuous} if it is so as a map
    \begin{equation*}
      X
      \xrightarrow{\quad}
      2^{\cS(\Gamma_b(\cE))}
      ,\quad
      \cS(\bullet)
      :=
      \left\{\text{continuous seminorms on }\bullet\right\}
    \end{equation*}
    for the pointwise topology on $\cS(\Gamma_b(\cE))$.
  \end{enumerate}
\end{definition}

Michael selection requires closed and convex values for multi-valued mappings; closure in the context of \Cref{th:subhom.lsc.unitariz}, for sets of seminorms, will always be with respect to the pointwise topology mentioned in \Cref{def:mv.renorm}\Cref{item:def:mv.renorm:lsc} as for convexity, cf. \cite[\S 1, p.17]{zbMATH03300430}:

\begin{definition}\label{def:lp.cvx}
  For $1\le p<\infty$ the \emph{$\ell^p$ convex structure} on the space of seminorms on a vector space is
  \begin{equation}\label{eq:lp.cvx}
    \lambda\|\cdot\|\oplus(1-\lambda)\vvvert\cdot\vvvert
    :=
    \left(\lambda \|\cdot\|^p+(1-\lambda)\vvvert\cdot\vvvert^p\right)^{\frac 1p}
    ,\quad
    \lambda\in [0,1].
  \end{equation}
\end{definition}

\begin{remarks}\label{res:l2cvx}
  \begin{enumerate}[(1),wide]
  \item The binary operations \Cref{eq:lp.cvx} parametrized by $\lambda\in [0,1]$ make the space $\cS(E)$ of seminorms on $E$ into a \emph{barycentric algebra} in the sense of \cite[\S 12.7]{schecht_hndbk-an}. Upon declaring a subset of $\cS(E)$ convex precisely when it is closed under all such operations, the resulting collection $\cC(E)\subset 2^{\cS(E)}$ of convex seminorm sets forms a \emph{convexity} (or convex structure \cite[\S I.1.1]{vdv_cvx}, hence the terminology) in the sense of \cite[Definition 2.1]{horv_top-cvx}:
    \begin{itemize}[wide]
    \item the empty set, $\cS(E)$ and singletons are convex;

    \item intersections of convex sets are convex;

    \item filtered unions of convex sets are convex. 
    \end{itemize}

  \item The fact that the $\ell^2$ convexity structure of \Cref{def:lp.cvx} induces one on the space of \emph{Hilbert} seminorms (i.e. those induced by inner products) will be implicit in much of the sequel.
  \end{enumerate}
\end{remarks}

With a view towards \Cref{th:subhom.lsc.unitariz} below, recall \cite[Defnition 3.1]{gierz_bdls} that \emph{locally paracompact} (here always assumed Hausdorff) spaces are those whose points all admit closed \emph{paracompact} \cite[\S 5.1]{eng_top_1989} neighborhoods. \cite[Theorem 3.2]{gierz_bdls} ensures that Banach bundles $\cE\xrightarrowdbl{}X$ over locally paracompact spaces are always \emph{full} in the sense of \cite[Definition 2.1]{gierz_bdls}:
\begin{equation*}
  \bigcup_{s\in \Gamma_b(\cE)}\im s = \cE. 
\end{equation*}
Or: every element of the total space $\cE$ is the image of some global bounded section. 

\begin{theorem}\label{th:subhom.lsc.unitariz}
  A subhomogeneous continuous Banach bundle over a locally paracompact space admits an equivalent lower semicontinuous multi-valued Hilbert renorming with non-empty, closed, $\ell^2$-convex values.
\end{theorem}
\begin{proof}
  To fix ideas and language, we focus on the real version of the result.

  Associating to a norm on a finite-dimensional real vector space $E$ (such as the fibers $\cE_x$) its unit ball produces a bijection \cite[Proposition 7.5]{trev_tvs} between norms and origin-symmetric \emph{convex bodies} (compact convex subsets of $E$ with non-empty interior). In that picture, the Hilbert norms on $E$ correspond precisely to \emph{ellipsoids} \cite[p.8]{gard_tom_2e_2006}: affine images of the standard unit ball or, more intrinsically, origin-symmetric convex bodies whose symmetry groups are maximal compact in $GL(E)$.

  We remind the reader that among all Hilbert norms on a finite-dimensional $E$ dominated by a given norm $\|\cdot\|$ there is one, $\vvvert\cdot\vvvert^L$, ``optimal'' in the sense that its associated (\emph{L\"owner} \cite[Theorem 9.2.1]{gard_tom_2e_2006}) ellipsoid $K^L\left(E_{\|\cdot\|\le 1}\right)$ has minimal volume among those containing the unit ball $E_{\|\cdot\|\le 1}$.

  We define the convex sets $\cK_{\bullet}$ of \Cref{eq:mv.renorm} recursively, proceeding upwards with respect to fiber dimensions. Throughout the argument, we conflate origin-symmetric convex bodies $K$ with their attached norms $\|\cdot\|_{K}$ \cite[\S 1.7.2, p.53]{schn_cvx_2e_2014} (and hence ellipsoids $K$ with Hilbert norms). The reader should furthermore feel free to ignore zero-dimensional fibers: those will play no role in the desired renorming in any case. Alternatively, one can always form the direct sum between the original $\cE$ and a rank-1 trivial bundle, so as to assume there are no 0-dimensional fibers to begin with.

  \begin{enumerate}[(I),wide]
  \item\label{item:th:subhomog.unitarzbl:gen} Consider first the closed locus $X_{0}\subseteq X$ where the fiber $\cE_x$ is locally minimal(-dimensional): 
    \begin{equation*}
      \dim\cE_{x'} \ge \dim\cE_x
      ,\quad\forall x'\in\text{ some neighborhood }U\ni x.
    \end{equation*}
    Over $X_0$, $\cK_x$ is defined as the singleton
    \begin{equation}\label{eq:step.0}
      \cK_0(\cE_{x,\|\cdot\|\le 1}):=\left\{K^L(\cE_{x,\|\cdot\|\le 1})\right\}
    \end{equation}
    consisting of the L\"owner ellipsoid of the original norm's unit ball. 

  \item In the next step of the recursion, define $X_1\subseteq X$ as the set of points $x\in X$ with
    \begin{equation*}
      \left(\exists\text{ nbhd }U\ni x\right)
      \left(\forall x'\in U\ :\ x'\in X_0\text{ or }\dim\cE_{x'}\ge \dim \cE_x\right).
    \end{equation*}
    
    For $x\in X_1\setminus X_0$, define $\cK_x$ as follows.
    \begin{itemize}[wide]
    \item For each {\it slice} 
      \begin{equation}\label{eq:slice}
        S=\cE_{x,\|\cdot\|\le 1}\bigcap \left(\text{non-zero linear subspace of $\cE_x$ of arbitrary dimension}\right)
      \end{equation}
      of the unit ball $\cE_{x,\|\cdot\|\le 1}\subset \cE_x$ consider the orthogonal complement $S^{\perp}\le \cE_x$ with respect to the inner product (induced by) the L\"owner ellipsoid $K^L(\cE_{x,\|\cdot\|\le 1})$.

    \item Form the L\"owner ellipsoid
      \begin{equation}\label{eq:ellips.s}
        K^L\left(\mathrm{co}\left(K^L(S)\cup \left(S^{\perp}\cap\cE_{x,\|\cdot\|\le 1}\right)\right)\right)
        ,\quad
        \mathrm{co}:=\text{convex hull}.
      \end{equation}
      
    \item Finally, set
      \begin{equation}\label{eq:kx.1}
        \cK_x
        :=
        \mathrm{co}_{\ell^2}\left(\bigcup_{\text{slices }S}\text{ellipsoid \Cref{eq:ellips.s}}\right)
      \end{equation}
    \end{itemize}
    with $\mathrm{co}_{\ell^2}$ denoting the hull for the $\ell^2$ convex structure of \Cref{def:lp.cvx}. The construction depending only on the unit ball $\cE_{x,\|\cdot\|\le 1}$ of the Banach space $\cE_x$, we also write (by analogy to \Cref{eq:step.0})
    \begin{equation}\label{eq:step.1}
      \cK_1(\cE_{x,\|\cdot\|\le 1})
      :=
      \cK_x\text{ as just defined}.
    \end{equation}
    I claim that the map
    \begin{equation*}
      X_1
      \ni x
      \xmapsto{\quad}
      \cK_x
    \end{equation*}
    defined thus far (on the closed subspace $X_1\subseteq X$) meets the requisite constraints.

    \begin{enumerate}[label={},wide]
    \item {\bf $\cK_x$ is LSC in $x\in X_1$.} The restrictions $(\cK_{x})|_{x\in X_0}$ and $(\cK_{x})|_{x\in X_1\setminus X_0}$ are even continuous as a consequence of \cite[Proposition 1.14]{2409.03531v1}, so fix a net
      \begin{equation*}
        X_1\setminus X_0
        \ni
        x_{\lambda}
        \xrightarrow[\lambda]{\quad\text{convergent}\quad}
        x\in X_0.
      \end{equation*}
      $\cK_x$ is in this case simply the singleton $\{K:=K^L(\cE_{x,\|\cdot\|\le 1})\}$, so it will suffice to prove that there are members $K_{\lambda}\in \cK_{x_{\lambda}}$ with
      \begin{equation*}
        \|s\|_{K_\lambda}
        \xrightarrow[\quad\lambda\quad]{}
        \|s\|_{K}
        ,\quad
        \forall \text{ local section $s$ around $x$}. 
      \end{equation*}
      To verify this:
      \begin{itemize}[wide]
      \item Extend a basis for $\cE_x$ to linearly independent local sections $(s_i)_i$ around $x$ (possible by local paracompactness: \cite[Theorem 3.2]{gierz_bdls} again).
        
      \item Those sections are linearly independent locally around $x$ by continuity \cite[Remarque preceding \S 2]{dd}.
      \item Finally, take for the desired $K_{\lambda}$ the L\"owner ellipsoids
      \begin{equation*}
        K_{\lambda}
        :=
        \text{\Cref{eq:ellips.s}}
        \in \cK_{x_{\lambda}}
        ,\quad
        S:=\mathrm{span}\left\{s_i(x_{\lambda})\right\}_i
      \end{equation*}
      (with $x_{\lambda}$ in place of $x$). 
      \end{itemize}
      
    \item {\bf Uniform bound on Banach-Mazur distances.} Over $X_0$ this is an immediate consequence of John's celebrated result (\cite[(2.1.4)]{MR1243006}, \cite[Proposition 9.12]{tj_bm}) that for finite-dimensional Banach spaces $(E,\|\cdot\|)$ we have
      \begin{equation*}
        d_{BM}\left(E_{\|\cdot\|\le 1},\ K^L\left(E_{\|\cdot\|\le 1}\right)\right)
        \le
        \sqrt{\dim E}.
      \end{equation*}
      As to
      \begin{equation*}
        \sup_{K\in\text{\Cref{eq:ellips.s}}} d_{BM}\left(\cE_{x,\|\cdot\|\le 1}, K\right)<\infty,
      \end{equation*}
      observe that for any finite-dimensional Banach space $(E,\|\cdot\|)$ there is a finite upper bound on
      \begin{equation*}
        C_S:=\inf\left\{C>0\ :\ C\cdot \mathrm{co}\left(S\cup \left( S^{\perp}\cap E_{\|\cdot\|\le 1}\right)\right)\supseteq E_{\|\cdot\|\le 1}\right\}
      \end{equation*}
      uniform over slices \Cref{eq:slice} and depending only on $\dim E$. Sprinkling L\"owner ellipsoids as in \Cref{eq:ellips.s} does not alter this uniform boundedness by John's theorem again. 
    \end{enumerate}
    
  \item To continue the recursive construction, assume $X_{\bullet}$ defined for $\bullet<k$, along with the restriction of the map \Cref{eq:mv.renorm} to that subspace of $X$:
    \begin{equation*}
      x\in X_{\bullet}\setminus X_{\bullet-1}
      \xRightarrow{\quad}
      \cK_x = \cK_{\bullet}(\cE_{x,\|\cdot\|\le 1})
    \end{equation*}
    for some $\cK_{\bullet}$ analogous to the $\cK_{0,1}$ of \Cref{eq:step.0} and \Cref{eq:step.1}. As perhaps expected, we next set 
    \begin{equation*}
      X_k:=
      \left\{x\in X\ |\ \left(\exists\text{ nbhd }U\ni x\right) \left(\forall x'\in U\ :\ x'\in X_{k-1}\text{ or }\dim\cE_{x'}\ge \dim \cE_x\right) \right\}.
    \end{equation*}
    To define
    \begin{equation*}
      \cK_x = \cK_k(\cE_{x,\|\cdot\|\le 1})
      \quad\text{for}\quad
      x\in X_{k}\setminus X_{k-1}
    \end{equation*}
    again extrapolate:
    \begin{itemize}[wide]
    \item For a slice \Cref{eq:slice} and each ellipsoid
      \begin{equation*}
        K\in \cK_{k-1}(S)\subset\left\{\text{origin-centered full-dimensional ellipsoids in }\mathrm{span}~S\right\},
      \end{equation*}
      form
      \begin{equation}\label{eq:ellips.s.gen}
        K^L\left(\mathrm{co}\left(K\cup \left(S^{\perp}\cap\cE_{x,\|\cdot\|\le 1}\right)\right)\right)
      \end{equation}
      with $S^{\perp}$ as before, the orthogonal complement with respect to $K^L(\cE_{x,\|\cdot\|\le 1})$.

    \item Then, mimicking \Cref{eq:kx.1}:
      \begin{equation*}
        \cK_x
        :=
        \cK_{k}(\cE_{x,\|\cdot\|\le 1})
        :=
        \mathrm{co}_{\ell^2}\left(\bigcup_{\substack{\text{slices }S\\K\in \cK_{k-1}(S)}}\text{ellipsoid \Cref{eq:ellips.s.gen}}\right).
      \end{equation*}
    \end{itemize}
    
    % % \begin{itemize}[wide]
    % % \item For each $x\in X$ and every subspace $L\le \cE_x$ consider the L\"owner ellipsoid $K^L(\cE_{x,\|\|\cdot\le 1}\cap L)$ in $L$.
    % %   
    % % \item Form the L\"owner ellipsoid 
    % % \end{itemize}
    % % 
    % % 
    % % \begin{equation*}
    % %   \cK(x)
    % %   :=
    % %   \overline{\mathrm{co}_{\ell^2}}
    % %   \left\{
    % %     K^L(\cE_{x,\|\cdot\|\le 1}\cap L)\ :\ \text{subspaces }L\le \cE_x
    % %   \right\}
    % % \end{equation*}
    % % where
    % % \begin{itemize}[wide]
    % % \item $\overline{\mathrm{co}_{\ell^2}}$ denotes the closed convex hull with respect to the $\ell^2$ convexity structure;
    % %   
    % % \item we identify origin-symmetric convex bodies with their corresponding norms;
    % %   
    % % \item and $K^L(\cE_{x,\|\cdot\|\le 1}\cap L)$ is (the norm associated to) the L\"owner ellipsoid of the unit ball of 
    % % \end{itemize}
    % % 
    
    The procedure eventually terminates by subhomogeneity, whereupon it will be no more difficult to prove than over $X_1$ that the resulting map \Cref{eq:mv.renorm} is LSC and exhibits the requisite Banach-Mazur boundedness.
  \end{enumerate}
\end{proof}

\begin{remark}\label{re:why.slices}
  The reason for including slices in the proof of \Cref{th:subhom.lsc.unitariz} is the fact that in general, the L\"owner ellipsoid of a slice is not contained in that of the ambient body.

  In $\bR^3$, say, the L\"owner ellipsoid of a cube is a standard ball. A rectangular slice through the diagonals of a pair of parallel faces, however, will have for its L\"owner ellipsoid an ellipse of {\it eccentricity} \cite[\S 6.3]{cg_geom-rev} $\frac 1{\sqrt 2}$; obviously, that ellipse is not contained in the ball circumscribed around the cube.
\end{remark}

\pf{th:subhom.unitariz}
\begin{th:subhom.unitariz}
  Consider a continuous subhomogeneous Banach bundle $\cE\xrightarrowdbl{\pi}X$ over a compact metrizable base, with fiber norms $\|\cdot\|_x$, $x\in X$. The claim amounts to proving the existence of a map
  \begin{equation*}
    X\ni x
    \xmapsto{\quad\vvvert\cdot\vvvert_{\bullet}\quad}
    \left(
      \text{Hilbert norm }
      \vvvert\cdot\vvvert_x
      \text{ on }
      \cE_x
    \right)
    \in
    \cS(\Gamma(\cE))
    :=
    \left\{\text{seminorms on }\Gamma(\cE)\right\}
  \end{equation*}
  continuous for the pointwise topology on the space of seminorms and with
  \begin{equation}\label{eq:renorm}
    \sup_{x\in X}
    d_{BM}\left(\|\cdot\|_x,\ \vvvert\cdot\vvvert_x\right)
    <\infty.
  \end{equation}
  \Cref{th:subhom.lsc.unitariz} provides a \emph{multi}-valued equivalent Hilbert renorming
  \begin{equation*}
    X
    \xrightarrow{\quad\cK\quad}
    2^{\cS(\Gamma(\cE))}
  \end{equation*}
  with weak$^*$-closed $\ell^2$-convex images. Because $X$ is metrizable \cite[Example 19.5(iii)]{gierz_bdls} shows that $\Gamma(\cE)$  is a separable Banach space (via \cite[Proposition 16.4]{gierz_bdls}, which ensures the example's hypothesis that the total space $\cE$ be Hausdorff).

  Observe next that the weak$^*$ topology on the space
  \begin{equation*}
    \cS_{\le K}(E)
    :=
    \left\{\text{seminorms }p\text{ on a Banach space $(E,\|\cdot\|)$}\ :\ p\le K\|\cdot\|\right\}
  \end{equation*}
  is metrizable when $E$ is separable: this is entirely analogous to and no more difficult to prove than the weak$^*$ metrizability \cite[Theorem V.5.1]{conw_fa} of the unit ball of $E$ under the same separability hypothesis. Indeed, the main ingredient there is the weak$^*$ compactness of the unit ball of a dual Banach space $E^*$ (\emph{Banach-Alaoglu}: \cite[Theorem V.3.1]{conw_fa}). To verify the parallel claim that
  \begin{equation*}
    \cS_{\le K}(E)
    \text{ is compact in the pointwise topology}
    ,\quad
    \forall\text{ Banach space }E,
  \end{equation*}
  simply observe that a net $(p_{\lambda})_{\lambda}$ in $\cS_{\le K}(E)$ has $\liminf_{\lambda} p_{\lambda}$ as a \emph{cluster point}; compactness follows \cite[Theorem 3.1.23]{eng_top_1989}. 
  
  Having concluded that the space $\cS_{\le K}\left(\Gamma(\cE)\right)$ is both compact and metrizable under our hypotheses, \cite[Theorem 3.4]{horv_top-cvx} applies: there is a continuous selection
  \begin{equation*}
    X\ni x
    \xmapsto{\quad}
    \vvvert\cdot\vvvert_x\in \cK_x.
  \end{equation*}
  The equivalence constraint \Cref{eq:renorm.mv} entails \Cref{eq:renorm}, and we are done. 
\end{th:subhom.unitariz}

%%%%%%%%%%%%%%%%%%%%%%%%%%%%%%%%%%%%%%%%%%%%%%%%%%%%%%%%%%%%%%%%%%%%%%%%%%%%%
%%%%%%%%%%%%%%%%%%%%%%%%%%%%%%%%%%%%%%%%%%%%%%%%%%%%%%%%%%%%%%%%%%%%%%%%%%%%%
\section{Complements on universal $n$-branching}\label{se:nbrnch}

The current section revolves around possible modes of failure for \Cref{th:subhom.unitariz} in the absence of metrizability. A first remark is that the counterexamples provided by \cite[Theorem A]{2409.17807v1} to the existence of finite-index expectations function for the same purpose here. We recall some of the background. 

For a set $X$ and $n\in \bZ_{>0}$ write
\begin{equation*}
  X_{[n]}
  :=
  \left\{A\in 2^X\ :\ 1\le |A|\le n\right\}
\end{equation*}
for its collection of non-empty subsets of cardinality $\le n$. This is what \cite[p.608]{zbMATH03569401} would denote by $X(n)$. Given a (typically Hausdorff) topology on $X$, the set $X_{[n]}$ acquires the \emph{Vietoris topology} of \cite[\S 1.2]{ct_vietoris}; it is also the quotient topology induced by the surjection
\begin{equation*}
  X^n
  \ni
  \left(x_i\right)_{i=1}^n  
  \xrightarrowdbl{\quad}
  \left\{x_i\right\}_{i=1}^n
  \in
  X_{[n]}.
\end{equation*}

To elucidate the section title's reference to universality, recall Pavlov and Troitskii's notion of \emph{branched cover} \cite[\S 1, p.338]{pt_brnch}: a continuous open surjection of compact Hausdorff spaces with a finite upper bound on the cardinalities of the fibers. We refer to such a gadget $Y\xrightarrowdbl{\pi}X$ as \emph{$(\le n)$-branched} (\emph{$n$-branched}) if
\begin{equation*}
  \max_{x\in X}\left|\pi^{-1}(x)\right|\le n
  \quad
  \left(\text{respectively }=n\right)
  \quad
  \left(\text{some $n\in \bZ_{>0}$}\right).
\end{equation*}

Writing (as on \cite[\S 1, p.4]{2409.17807v1})
\begin{equation*}
  Y_{[n]}^{\subseteq}
  :=
  \left\{(y,A)\in Y\times Y_{[n]}\ :\ y\in A\right\},
\end{equation*}
the map
\begin{equation*}
  Y_{[n]}^{\subseteq}
  \xrightarrowdbl[\quad\text{second projection}\quad]{\quad\pi_Y\quad}
  Y_{[n]}
\end{equation*}
is an $n$-branched cover. Collectively (for varying compact Hausdorff $Y$), the $\pi_Y$ form a family roughly speaking ``as ill-behaved as $(\le n)$-branched covers can be'': the aforementioned universality. The following simple remark formalizes the as-yet-vague principle. 

\begin{lemma}\label{le:nbrnch.univ}
  Every $(\le n)$-branched cover $Y\xrightarrowdbl{\pi}X$ fits into a \emph{pullback} \cite[Definition 11.8]{ahs}
  \begin{equation*}
    \begin{tikzpicture}[>=stealth,auto,baseline=(current  bounding  box.center)]
      \path[anchor=base] 
      (0,0) node (l) {$Y$}
      +(2,.5) node (u) {$Y_{[n]}^{\subseteq}$}
      +(2,-.5) node (d) {$X$}
      +(4,0) node (r) {$Y_{[n]}$}
      ;
      \draw[->] (l) to[bend left=6] node[pos=.5,auto] {$\scriptstyle $} (u);
      \draw[->] (u) to[bend left=6] node[pos=.5,auto] {$\scriptstyle \pi_Y$} (r);
      \draw[->] (l) to[bend right=6] node[pos=.5,auto,swap] {$\scriptstyle \pi$} (d);
      \draw[->] (d) to[bend right=6] node[pos=.5,auto,swap] {$\scriptstyle \iota_{\pi}$} (r);
    \end{tikzpicture}
  \end{equation*}
  in the category of compact Hausdorff spaces for the embedding
  \begin{equation*}
    X\ni x
    \xmapsto{\quad\iota_{\pi}\quad}
    \pi^{-1}(x)
    \in
    Y_{[n]}
  \end{equation*}
\end{lemma}
\begin{proof}
  There is very little to check; the main point (itself an unwinding of the definitions) is perhaps that the continuity and openness of $\pi$ entail the continuity of $\iota_{\pi}$ in the \emph{upper} and \emph{lower Vietoris topologies} of \cite[\S 1.2]{ct_vietoris} respectively, so that $\iota_{\pi}$ is indeed Vietoris-continuous. 
\end{proof}

%% REPRESENTABLE FUNCTOR: START

The incipient category-theoretic picture suggested by \Cref{le:nbrnch.univ} can be expanded. Some terminology will aid that goal. %The ensuing discussion relies on categorical rudiments such as representable functors, (co)limits, functor (co)continuity and the like, covered in any number of sources: \cite{mcl_2e,bw,ahs}, etc.

\begin{definition}\label{def:anch}
  Let $Y\xrightarrow{\pi}X$ be a map and $Z$ a set. A \emph{$Z$-anchor} for $\pi$ is a map $Y\xrightarrow{f}Z$ injective on every fiber $\pi^{-1}(x)$, $x\in X$. We also refer to such pairs $(\pi,f)$ as \emph{$Z$-anchored maps $Y\to X$}. 

  The same terminology applies to topological spaces and continuous maps or, more generally, to any \emph{($\cat{Set}$-)concrete category} in the sense of \cite[Definition 5.1]{ahs}: a pair $(\cC,U)$ with  
  \begin{equation*}
    \text{category }\cC
    \xrightarrow[\quad\text{faithful functor}\quad]{\quad U\quad}
    \cat{Set}
    :=
    \text{category of sets}.
  \end{equation*}
  The one example of interest will be the category $\cat{Cpct}_{T_2}$ of compact Hausdorff spaces with its obvious forgetful functor to $\cat{Set}$. 
\end{definition}

\begin{theorem}\label{th:repr.anchors}
  For any compact Hausdorff space $Z$ and $n\in \bZ_{>0}$ the space $Z_{[n]}$ represents the functor
  \begin{equation}\label{eq:anch.cov.functor}
    \cat{Cpct}_{T_2}^{op}
    \ni
    X
    \xmapsto{\quad\cat{BC}^Z_{\le n}\quad}
    \left(\text{$Z$-anchored $(\le n)$-branched covers of $X$}\right)
    \in
    \cat{Set}.    
  \end{equation}
\end{theorem}
\begin{proof}
  Observe first that there is a upper bound on the cardinalities of the domains of an $(\le n)$-branched cover of any given $X\in \cat{Cpct}_{T_2}$, so there are no set-theoretic issues: \Cref{eq:anch.cov.functor} does indeed take values in \emph{sets} as opposed to \emph{proper classes} \cite[\S I.6, p.23]{mcl_2e}. Furthermore, $\cat{BC}^Z_{\le n}$ operates on morphisms by pullback: pulling back an $(\le n)$-branched cover over any morphism in $\cat{Cpct}_{T_2}$ produces another such.

  We exhibit back-and-forth natural transformations between $\cat{BC}^Z_{\le n}$ and the representable functor $\cat{Cpct}_{T_2}\left(-,Z_{[n]}\right)$.

  \begin{enumerate}[(I),wide]
  \item\label{item:th:repr.anchors:pf.bc2rep} \textbf{: Branched covers $\to$ morphisms into $Z_{[n]}$.} Suppose
    \begin{equation}\label{eq:init.anch}
      \begin{tikzpicture}[>=stealth,auto,baseline=(current  bounding  box.center)]
        \path[anchor=base] 
        (0,0) node (l) {$X$}
        +(2,.5) node (u) {$Y$}
        +(4,0) node (r) {$Z$}
        ;
        \draw[<<-] (l) to[bend left=6] node[pos=.5,auto] {$\scriptstyle \pi$} (u);
        \draw[->] (u) to[bend left=6] node[pos=.5,auto] {$\scriptstyle f$} (r);
      \end{tikzpicture}
    \end{equation}
    is a $Z$-anchored $(\le n)$-branched cover. Once more, as in the proof of \Cref{le:nbrnch.univ}, the openness and continuity of $\pi$ ensure the Vietoris-continuity of
    \begin{equation*}
      X
      \ni
      x
      \xmapsto{\quad}
      \pi^{-1}(x)
      \in
      Y_{[n]};
    \end{equation*}
    further composed with the map $Y_{[n]}\xrightarrow{f_{[n]}}Z_{[n]}$ induced by $f$, this yields the desired $\cat{Cpct}_{T_2}$-morphism $X\to Z_{[n]}$. 
    
  \item\label{item:th:repr.anchors:pf.rep2bc} \textbf{: Morphisms into $Z_{[n]}$ $\to$ branched covers.} Fit a morphism $X\xrightarrow{g}Z_{[n]}$ into the lower left-hand pullback
    \begin{equation}\label{eq:y.as.plbk}
      \begin{tikzpicture}[>=stealth,auto,baseline=(current  bounding  box.center)]
        \path[anchor=base] 
        (0,0) node (l) {$Y$}
        +(2,.5) node (u) {$Z_{[n]}^{\subseteq}$}
        +(2,-.5) node (d) {$X$}
        +(4,0) node (r) {$Z_{[n]}$}
        +(6,0) node (rr) {$Z$,}
        ;
        \draw[->] (l) to[bend left=6] node[pos=.5,auto] {$\scriptstyle $} (u);
        \draw[->] (u) to[bend left=6] node[pos=.5,auto,swap] {$\scriptstyle \pi_Z$} (r);
        \draw[->] (l) to[bend right=6] node[pos=.5,auto,swap] {$\scriptstyle \pi$} (d);
        \draw[->] (d) to[bend right=6] node[pos=.5,auto,swap] {$\scriptstyle g$} (r);
        \draw[->] (u) to[bend left=20] node[pos=.5,auto,sloped] {$\scriptstyle \text{first projection}$} (rr);
        \draw[->] (l) to[bend left=60] node[pos=.5,auto,sloped] {$\scriptstyle f$} (rr);
      \end{tikzpicture}
    \end{equation}
    with the anchor $f$ defined as the pictured composition. 
    
  \end{enumerate}
  The verification that the two natural transformations \Cref{item:th:repr.anchors:pf.bc2rep} and \Cref{item:th:repr.anchors:pf.rep2bc} compose to identities in both directions is not difficult. On the one hand
  \begin{equation*}
    \text{\Cref{item:th:repr.anchors:pf.bc2rep}}\circ\text{\Cref{item:th:repr.anchors:pf.rep2bc}} = \id_{\cat{Cpct}_{T_2}(X,Z_{[n]})}
  \end{equation*}
  virtually tautologically: in \Cref{eq:y.as.plbk}, $g=f_{[n]}\circ \pi^{-1}$ by direct examination. Conversely, given $Z$-anchored branched cover \Cref{eq:init.anch}, the canonical map from $Y$ to the pullback in the lower left of \Cref{eq:y.as.plbk} is on the one hand injective by the anchoring requirement (an element $y\in Y$ is uniquely defined by $(\pi(y),f(y))$), and on the other hand onto by the surjectivity of $\pi$. 
\end{proof}

%% REPRESENTABLE FUNCTOR: END

In other words, \Cref{th:repr.anchors} says that
\begin{equation*}
  \begin{tikzpicture}[>=stealth,auto,baseline=(current  bounding  box.center)]
    \path[anchor=base] 
    (0,0) node (l) {$Z_{[n]}$}
    +(2,.5) node (u) {$Z^{\subseteq}_{[n]}$}
    +(4,0) node (r) {$Z$}
    ;
    \draw[<<-] (l) to[bend left=6] node[pos=.5,auto] {$\scriptstyle \pi_Z$} (u);
    \draw[->] (u) to[bend left=6] node[pos=.5,auto] {$\scriptstyle \text{first projection}$} (r);
  \end{tikzpicture}
\end{equation*}

is the universal $Z$-anchored $(\le n)$-branched cover. It is this ``maximal misbehavior'' of the $\pi_Z$ that makes those maps good candidates for the counterexamples provided by \cite[Theorem A]{2409.17807v1} to \cite[Theorem 1.1, (1) $\Rightarrow$ (2)]{pt_brnch}. The same phenomenon manifests in the present context as Hilbert-renorming failure.

\Cref{pr:zn.not.hilb.renorm} and the subsequent discussion take for granted the correspondence (\cite[Propositions 1.2 and 1.3]{zbMATH04054348} or \cite[Scholium 6.7]{hk_shv-bdl} more broadly, beyond the $C^*$ framework) between continuous unital $C^*$ fields over compact Hausdorff $X$ continuous (in the appropriate sense) unital \emph{$C(X)$-algebras} \cite[Definition 1.5]{zbMATH04056334}: unital $C^*$ embeddings
\begin{equation*}
  C(X)
  \lhook\joinrel\xrightarrow{\quad}
  Z(A):=\text{center of the $C^*$-algebra }A.
\end{equation*}

\begin{proposition}\label{pr:zn.not.hilb.renorm}
  If the compact Hausdorff space $Z$ contains a copy of the \emph{one-point compactification} \cite[Theorem 3.5.11]{eng_top_1989} of an uncountable discrete space and $n\in \bZ_{\ge 3}$ then the $C^*$ bundle over $Z_{[n]}$ associated to the embedding
  \begin{equation*}
    C\left(Z_{[n]}\right)
    \lhook\joinrel\xrightarrow{\quad\pi_Z^*\quad}
    C\left(Z_{[n]}^{\subseteq}\right)
  \end{equation*}
  dual to $\pi_Z$ does not admit a Hilbert renorming. 
\end{proposition}
\begin{proof}
  As recalled in the introductory discussion, \cite[Theorem A]{2409.17807v1} argues that under the hypotheses there is no finite-index conditional expectation  $C\left(Z_{[n]}^{\subseteq}\right)\to C\left(Z_{[n]}\right)$. That the present statement is effectively a rephrasing of that remark follows from \Cref{le:expect.renorm}. 
\end{proof}

\begin{lemma}\label{le:expect.renorm}
  Let $X$ be a compact Hausdorff space and $C(X)\lhook\joinrel\xrightarrow{\iota}A$ a continuous unital $C(X)$-algebra.

  The associated $C^*$ bundle $\cA\xrightarrowdbl{} X$ admits a Hilbert renorming if and only if there is a finite-index expectation $A\xrightarrowdbl{E}C(X)$. 
\end{lemma}
\begin{proof}
  The implication $(\Leftarrow)$ is (a very small) part of \cite[Theorem 1]{fk_fin-ind}: given the expectation $E$ inducing states $E_x$ on the fibers $\cA_x$, $x\in X$, the family of Hilbert norms
  \begin{equation*}
    \vvvert a\vvvert_x
    :=
    E_x(a^*a)^{1/2}
    ,\quad
    a\in \cA_x
  \end{equation*}
  will do. Conversely, a Hilbert renorming will make $A=\Gamma(\cA)$ into a \emph{Hilbert module} \cite[p.4]{lnc_hilb} over $C(X)$ for an inner product
  \begin{equation}\label{eq:glob.inn}
    A\times A
    \xrightarrow{\quad\braket{-\mid -}\quad}
    C(X)
  \end{equation}
  (cf. \cite[Theorem 2.11]{zbMATH03738314} for the correspondence between Hilbert bundles and modules). The desired expectation will be
  \begin{equation}\label{eq:e.from.renorm}
    E(a)
    :=
    \frac{\braket{1\mid a}}{\braket{1\mid 1}}    
    ,\quad
    a\in A,
  \end{equation}
  which does make sense: the global inner product \Cref{eq:glob.inn} is by assumption fiber-wise non-degenerate, so the denominator is a strictly positive element of $C(X)$. The finite-index requirement follows easily from the assumed Banach-Mazur boundedness \Cref{eq:renorm.mv}. 
\end{proof}

% % \begin{lemma}\label{le:inv.11}
% %   Let $X$ be a compact Hausdorff space and $\braket{-\mid-}$ a self-Hilbert module structure on $C(X)$ whose underlying norm
% %   \begin{equation*}
% %     \vvvert a\vvvert
% %     :=
% %     \left\|\braket{a\mid a}\right\|^{1/2}
% %   \end{equation*}
% %   is equivalent to the original supremum norm $\|\cdot\|$ on $C(X)$. The positive element $\braket{1\mid 1}\in C(X)$ is then invertible. 
% % \end{lemma}
% % \begin{proof}
% % 
% %   \mgnt{do}
% % \end{proof}

%%%%%%%%%%%%%%%%%%%%%%%%%%%%%%%%%%%%%%%%%%%%%%%%%%%%%%%%%%%%%%%%%%%%%%%%%%%%%
\subsection{Asides on the literature}\label{subse:asides.xn.magerl}

The current subsection retraces some of the material of \cite{zbMATH03569401}, which (following \cite{MR450972}) leverages the spaces $X_{[n]}$ to provide counterexamples to Michael selection \cite[Theorem 3.2'']{mich_contsel-1} in the absence of metrizability. The goal is to generalize some of the auxiliary results slightly, as well as to correct what appear to me to be a few (small) errors: see \hyperref[res:magerl.no.lin.fn]{Remarks~\ref*{res:magerl.no.lin.fn}}. 

Recall \cite[preceding \S 3]{zbMATH03569401} that for a subset $X\subseteq E$ of a (real or complex) vector space a \emph{convex selection}
\begin{equation*}
  X_{[n]}
  \xrightarrow{\quad\phi\quad}
  E
\end{equation*}
is a map with
\begin{equation*}
  \phi(A)
  \in
  \mathrm{co}(A)
  :=
  \text{\emph{convex hull} of $A$ \cite[\S 7, p.57]{trev_tvs}}
  ,\quad
  \forall A\in X_{[n]}. 
\end{equation*}
When $E$ is a (again, usually Hausdorff) topological vector space continuity for such selections is always understood in terms of the Vietoris topology on the domain.

The intent behind \cite[Lemma 3]{zbMATH03569401} is to ``slice'' a convex selection on $X_{[n]}$ into ``lower'' convex selections, preserving appropriate continuity conditions. We revisit the result, removing the original constraint that $n=2,3$.

\begin{lemma}\label{le:slice.conv.sel}
  Let $n\ge 2$ be a positive integer and
  \begin{equation}\label{eq:phi.sel}
    X_{[n]}
    \xrightarrow{\quad\phi\quad}
    E
    ,\quad
    X\subseteq\text{topological vector space }E
  \end{equation}
  a convex selection continuous at $\left\{x,x_0\right\}\in X_{[n]}$.

  There are
  \begin{equation*}
    X_{[n-1]}
    \xrightarrow{\quad f^x\quad}
    [0,1]
    \quad\text{and}\quad
    X_{[n-1]}
    \xrightarrow[\text{convex selection}]{\quad\phi^x\quad}
    E
  \end{equation*}
  such that
  \begin{equation}\label{eq:slice.phi}
    \phi(A\cup \{x\})
    =
    f^x(A)x  
    +
    (1-f^x(A))\phi^x(A)
    ,\quad
    \forall A\in X_{[n-1]}.
  \end{equation}
  \begin{enumerate}[(1),wide]
  \item\label{item:le:slice.conv.sel:phicont} $\phi^x$ is continuous at $\left\{x_0\right\}$.
    
  \item\label{item:le:slice.conv.sel:fcont} If $x\ne x_0$ then $f^x$ is continuous at $\left\{x_0\right\}$.    
  \end{enumerate}
\end{lemma}
\begin{proof}
  Existence is plain enough: $\phi(A\cup \{x\})$ is by assumption a convex combination of elements of $A\cup \{x\}$, and hence a convex combination of $x$ and some $\phi^x(A)\in \mathrm{co}(A)$. We henceforth focus on the continuity claims.

  \begin{enumerate}[(1),wide]

  \item Convex selections \Cref{eq:phi.sel} generally, for Hausdorff topological vector spaces $E$, are automatically continuous at singletons (this applies to $\phi$ and $\phi^x$ alike). Indeed, given convergent nets \cite[\S 1.6]{eng_top_1989}
    \begin{equation*}
      x_{i,\lambda}
      \xrightarrow[\quad\lambda\quad]{}
      x_0
      ,\quad
      1\le i\le n
    \end{equation*}
    and
    \begin{equation*}
      \left(c_{i,\lambda}\right)_{i=1}^n\in \Delta^{n-1}
      :=
      \text{$(n-1)$-simplex }
      \left\{(t_i)_{i=1}^n\in [0,1]^n\ :\ \sum t_i=1\right\}
    \end{equation*}
    the compactness of $\Delta^{n-1}$ allows us \cite[Theorem 3.1.23]{eng_top_1989} to assume convergence
    \begin{equation*}
      \left(c_{i,\lambda}\right)_{i=1}^n
      \xrightarrow[\quad\lambda\quad]{}
      \left(c_{i}\right)_{i=1}^n
      \in \Delta^{[n-1]},      
    \end{equation*}
    whence
    \begin{equation*}
      \sum_{i=1}^n c_{i,\lambda}x_{i,\lambda}
      \xrightarrow[\quad\lambda\quad]{}
      \sum_{i=1}^n c_{i}x_{0}
      =
      x_0.
    \end{equation*}
    
  \item Suppose
    \begin{equation*}
      A_{\lambda}
      \xrightarrow[\quad\lambda\quad]{\quad\text{Vietoris-convergent}\quad}
      \left\{x_0\right\}
      \in
      X_{[n-1]}
      \quad
      \xRightarrow{\quad\text{part \Cref{item:le:slice.conv.sel:phicont}}\quad}
      \quad
      \phi^x(A_{\lambda})
      \xrightarrow[\quad\lambda\quad]{}
      x_0.
    \end{equation*}
    By \Cref{eq:slice.phi} and the assumed continuity of $\phi$ at $\{x,x_0\}$ we have
    \begin{equation*}
      t_{\lambda}x+(1-t_{\lambda})x_{\lambda}
      \xrightarrow[\quad\lambda\quad]{}
      t_0x+(1-t_0)x_0
      \quad\text{for}\quad
      \begin{aligned}
        x_{\lambda}&:=\phi^x(A_{\lambda})\\
        t_{\lambda}&:=f^x(A_{\lambda})\\
        t_0&:=f^x(\{x_0\})
      \end{aligned}            
    \end{equation*}    
    Subtracting $x$ from both sides produces
    \begin{equation*}
      (1-t_{\lambda})(x_{\lambda}-x)
      \xrightarrow[\quad\lambda\quad]{}
      (1-t_{0})(x_{0}-x),
    \end{equation*}
    whence the conclusion: we are assuming $x\ne x_0$ and have already observed that $x_{\lambda}\xrightarrow[\lambda]{}x_0$, so the coefficients must converge as expected as well: $t_{\lambda}\xrightarrow[\lambda]{}t_0$. 
  \end{enumerate}
\end{proof}

\begin{remarks}\label{res:magerl.no.lin.fn}
  \begin{enumerate}[(1),wide]
  \item\label{item:res:magerl.no.lin.fn:gaps} As originally stated and proved, \cite[Lemma 3]{zbMATH03569401} appears to me to contain two gaps:

    \begin{itemize}[wide]
    \item The statement assumes the continuity of $\phi$ at $\{x_0\}$, but the \emph{proof} implicitly assumes it at $\{x,x_0\}$ instead. It is not difficult to produce examples disproving the claim in its original form: take $n:=2$, set $\phi(\left\{x',x_0\right\}):=x_0$ for $x'$ sufficiently close to $x_0$, and define it arbitrarily otherwise. There is no reason, then, why the coefficients in \Cref{eq:slice.phi} would be continuous at $x_0$ for $x$ sufficiently far from $x_0$.

    \item The result is stated for arbitrary (Hausdorff) topological vector spaces $E$, and the proof begins by selecting a continuous linear functional $E\to \bR$ non-vanishing on a certain non-zero element. This assumes such continuous linear functionals exist, which they don't always: the topological vector spaces $L^p$, $0<p<1$ of \cite[\S 15.9(8)]{k_tvs-1} admit no non-zero continuous functionals at all \cite[\S 15.9(9)]{k_tvs-1}. For the proof to go through as-is, one must assume the existence of enough continuous linear functionals on $E$ (\emph{local convexity} suffices for instance, by \emph{Hahn-Banach} \cite[p.187, Corollary 2]{trev_tvs}). 
    \end{itemize}

  \item\label{item:res:magerl.no.lin.fn:autocont} Per the proof of \Cref{le:slice.conv.sel}\Cref{item:le:slice.conv.sel:phicont} above, no additional assumptions (such as the requirement of \cite[Lemma 3]{zbMATH03569401} that $\phi\left(\left\{x,x_0\right\}\right)\not\in \{x,x_0\}$) are needed in order to prove continuity at singletons: it is automatic for all convex selections.

    In reference to item \Cref{item:res:magerl.no.lin.fn:gaps} above, it would (for the same reason) not be necessary to \emph{assume} continuity of the original $\phi$ at $\{x_0\}$.
  \end{enumerate}
\end{remarks}

% % OLD: NO LONGER NEEDED
% % 
% % The following simple observation will obviate the difficulty pointed out in \Cref{res:magerl.no.lin.fn}, confirming that \Cref{le:slice.conv.sel} (and hence also \cite[Lemma 3]{zbMATH03569401}) does hold regardless of local convexity. 
% % 
% % \begin{lemma}\label{le:net.cvx}
% %   Let
% %   \begin{equation*}
% %     (x_{\lambda})_{\lambda}
% %     ,\
% %     (y_{\lambda})_{\lambda}
% %     \subset \text{topological vector space }E
% %     \quad\text{and}\quad
% %     (t_{\lambda})_{\lambda}
% %     \subset
% %     [0,1]
% %   \end{equation*}
% %   be three nets. If
% %   \begin{equation*}
% %     x_{\lambda}
% %     \xrightarrow[\quad\lambda\quad]{}
% %     x
% %     \ne
% %     y
% %     \xleftarrow[\quad\lambda\quad]{}
% %     y_{\lambda}
% %     \quad\text{and}\quad
% %     t_{\lambda}x_{\lambda}+(1-t_{\lambda})y_{\lambda}
% %     \xrightarrow[\quad\lambda\quad]{}
% %     z
% %   \end{equation*}
% %   then 
% % \end{lemma}
% % \begin{proof}
% %   
% %   \mgnt{do}
% % \end{proof}

%%%%%%%%%%%%%%%%%%%%%%%%%%%%%%%%%%%%%%%%%%%%%%%%%%%%%%%%%%%%%%%%%%%%%%%%%%%%%
%%%%%%%%%%%%%%%%%%%%%%%%%%%%%%%%%%%%%%%%%%%%%%%%%%%%%%%%%%%%%%%%%%%%%%%%%%%%%

\addcontentsline{toc}{section}{References}
%\bibliography{bib}{}

\begin{thebibliography}{10}
  
\bibitem{ahs}
Ji{\v{r}}{\'{\i}} Ad{\'a}mek, Horst Herrlich, and George~E. Strecker.
\newblock Abstract and concrete categories: the joy of cats.
\newblock {\em Repr. Theory Appl. Categ.}, 2006(17):1--507, 2006.

\bibitem{blk}
B.~Blackadar.
\newblock {\em Operator algebras}, volume 122 of {\em Encyclopaedia of
  Mathematical Sciences}.
\newblock Springer-Verlag, Berlin, 2006.
\newblock Theory of $C^*$-algebras and von Neumann algebras, Operator Algebras
  and Non-commutative Geometry, III.

\bibitem{bg_cx-exp}
Etienne Blanchard and Ilja Gogi\'{c}.
\newblock On unital {$C(X)$}-algebras and {$C(X)$}-valued conditional
  expectations of finite index.
\newblock {\em Linear Multilinear Algebra}, 64(12):2406--2418, 2016.

\bibitem{2409.17807v1}
Alexandru Chirvasitu.
\newblock Finite-index phenomena and the topology of bundle singularities,
  2024.
\newblock \url{http://arxiv.org/abs/2409.17807v1}.

\bibitem{2409.03531v1}
Alexandru Chirvasitu.
\newblock Non-commutative branched covers and bundle unitarizability, 2024.
\newblock \url{http://arxiv.org/abs/2409.03531v1}.

\bibitem{ct_vietoris}
Maria~Manuel Clementino and Walter Tholen.
\newblock A characterization of the {Vietoris} topology.
\newblock {\em Topol. Proc.}, 22:71--95, 1997.

\bibitem{conw_fa}
John~B. Conway.
\newblock {\em A course in functional analysis}, volume~96 of {\em Graduate
  Texts in Mathematics}.
\newblock Springer-Verlag, New York, second edition, 1990.

\bibitem{cg_geom-rev}
H.~S.~M. Coxeter and S.~L. Greitzer.
\newblock {\em Geometry revisited}, volume~19 of {\em New Math. Libr.}
\newblock The Mathematical Association of America, Washington, DC, 1967.

\bibitem{dd}
Jacques Dixmier and Adrien Douady.
\newblock Champs continus d'espaces hilbertiens et de {$C^{\ast} $}-alg\`ebres.
\newblock {\em Bull. Soc. Math. France}, 91:227--284, 1963.

\bibitem{dg_ban-bdl}
M.~J. Dupr{\'e} and R.~M. Gillette.
\newblock {\em Banach bundles, {Banach} modules and automorphisms of
  {{\(C^*\)}}-algebras}, volume~92 of {\em Res. Notes Math., San Franc.}
\newblock Pitman Publishing, London, 1983.

\bibitem{eng_top_1989}
Ryszard Engelking.
\newblock {\em General topology.}, volume~6 of {\em Sigma Ser. Pure Math.}
\newblock Berlin: Heldermann Verlag, rev. and compl. ed. edition, 1989.

\bibitem{fd_bdl-1}
J.~M.~G. Fell and R.~S. Doran.
\newblock {\em Representations of *-algebras, locally compact groups, and
  {Banach} *- algebraic bundles. {Vol}. 1: {Basic} representation theory of
  groups and algebras}, volume 125 of {\em Pure Appl. Math., Academic Press}.
\newblock Boston, MA etc.: Academic Press, Inc., 1988.

\bibitem{zbMATH03300430}
W.~J. Firey.
\newblock p-means of convex bodies.
\newblock {\em Math. Scand.}, 10:17--24, 1962.

\bibitem{fk_fin-ind}
Michael Frank and Eberhard Kirchberg.
\newblock On conditional expectations of finite index.
\newblock {\em J. Oper. Theory}, 40(1):87--111, 1998.

\bibitem{gard_tom_2e_2006}
Richard~J. Gardner.
\newblock {\em Geometric tomography}, volume~58 of {\em Encycl. Math. Appl.}
\newblock Cambridge: Cambridge University Press, 2nd ed. edition, 2006.

\bibitem{gierz_bdls}
Gerhard Gierz.
\newblock {\em Bundles of topological vector spaces and their duality}, volume
  955 of {\em Lect. Notes Math.}
\newblock Springer, Cham, 1982.

\bibitem{gog_top-fg}
Ilja Gogi{\'c}.
\newblock Topologically finitely generated {Hilbert} {{\(C(X)\)}}-modules.
\newblock {\em J. Math. Anal. Appl.}, 395(2):559--568, 2012.

\bibitem{MR3056657}
Ilja Gogi\'{c}.
\newblock On derivations and elementary operators on {$C^*$}-algebras.
\newblock {\em Proc. Edinb. Math. Soc. (2)}, 56(2):515--534, 2013.

\bibitem{hk_shv-bdl}
Karl~Heinrich Hofmann and Klaus Keimel.
\newblock Sheaf theoretical concepts in analysis: {Bundles} and sheaves of
  {Banach} spaces, {Banach} {C}({X})-modules.
\newblock Applications of sheaves, {Proc}. {Res}. {Symp}., {Durham} 1977,
  {Lect}. {Notes} {Math}. 753, 415-441 (1979)., 1979.

\bibitem{horv_top-cvx}
Charles~D. Horvath.
\newblock Topological convexities, selections and fixed points.
\newblock {\em Topology Appl.}, 155(8):830--850, 2008.

\bibitem{zbMATH04056334}
G.~G. Kasparov.
\newblock Equivariant {KK}-theory and the {Novikov} conjecture.
\newblock {\em Invent. Math.}, 91(1):147--201, 1988.

\bibitem{kt_corresp}
Erwin Klein and Anthony~C. Thompson.
\newblock {\em Theory of correspondences. {Including} applications to
  mathematical economics}.
\newblock Can. Math. Soc. Ser. Monogr. Adv. Texts. John Wiley, New York, NY,
  1984.

\bibitem{k_tvs-1}
Gottfried K\"{o}the.
\newblock {\em Topological vector spaces. {I}}.
\newblock Die Grundlehren der mathematischen Wissenschaften, Band 159.
  Springer-Verlag New York, Inc., New York, 1969.
\newblock Translated from the German by D. J. H. Garling.

\bibitem{lnc_hilb}
E.~C. Lance.
\newblock {\em Hilbert {$C^*$}-modules}, volume 210 of {\em London Mathematical
  Society Lecture Note Series}.
\newblock Cambridge University Press, Cambridge, 1995.
\newblock A toolkit for operator algebraists.

\bibitem{MR1243006}
J.~Lindenstrauss and V.~D. Milman.
\newblock The local theory of normed spaces and its applications to convexity.
\newblock In {\em Handbook of convex geometry, {V}ol.\ {A}, {B}}, pages
  1149--1220. North-Holland, Amsterdam, 1993.

\bibitem{mcl_2e}
Saunders Mac~Lane.
\newblock {\em Categories for the working mathematician}, volume~5 of {\em
  Graduate Texts in Mathematics}.
\newblock Springer-Verlag, New York, second edition, 1998.

\bibitem{zbMATH03569401}
G.~M{\"a}gerl.
\newblock Metrizability of compact sets and continuous selections.
\newblock {\em Proc. Am. Math. Soc.}, 72:607--612, 1978.

\bibitem{mich_contsel-1}
Ernest Michael.
\newblock Continuous selections. {I}.
\newblock {\em Ann. Math. (2)}, 63:361--382, 1956.

\bibitem{pt_brnch}
A.~A. Pavlov and E.~V. Troitskii.
\newblock Quantization of branched coverings.
\newblock {\em Russ. J. Math. Phys.}, 18(3):338--352, 2011.

\bibitem{pvl_disc}
Miroslav Pavlovi{\'c}.
\newblock {\em Function classes on the unit disc. {An} introduction}, volume~52
  of {\em De Gruyter Stud. Math.}
\newblock Berlin: De Gruyter, 2nd revised and extended edition edition, 2019.

\bibitem{zbMATH06329568}
Du{\v{s}}an Repov{\v{s}} and Pavel~V. Semenov.
\newblock Continuous selections of multivalued mappings.
\newblock In {\em Recent progress in general topology III. Based on the
  presentations at the Prague symposium, Prague, Czech Republic, 2001}, pages
  711--749. Amsterdam: Atlantis Press, 2014.

\bibitem{rs_cont-sel}
Du{\v{s}}an Repov{\v{s}} and Pavel~Vladimirovi{\v{c}} Semenov.
\newblock {\em Continuous selections of multivalued mappings}, volume 455 of
  {\em Math. Appl., Dordr.}
\newblock Dordrecht: Kluwer, 1998.

\bibitem{zbMATH04054348}
Marc~A. Rieffel.
\newblock Continuous fields of {{\(C^*\)}}-algebras coming from group cocycles
  and actions.
\newblock {\em Math. Ann.}, 283(4):631--643, 1989.

\bibitem{schecht_hndbk-an}
Eric Schechter.
\newblock {\em Handbook of analysis and its foundations}.
\newblock San Diego, CA: Academic Press, 1997.

\bibitem{schn_cvx_2e_2014}
Rolf Schneider.
\newblock {\em Convex bodies: the {Brunn}-{Minkowski} theory}, volume 151 of
  {\em Encycl. Math. Appl.}
\newblock Cambridge: Cambridge University Press, 2nd expanded ed. edition,
  2014.

\bibitem{zbMATH03738314}
Alonso Takahashi.
\newblock A duality between {Hilbert} modules and fields of {Hilbert} spaces.
\newblock {\em Rev. Colomb. Mat.}, 13:93--120, 1979.

\bibitem{tj_bm}
Nicole Tomczak-Jaegermann.
\newblock {\em Banach-{Mazur} distances and finite-dimensional operator
  ideals}, volume~38 of {\em Pitman Monogr. Surv. Pure Appl. Math.}
\newblock Harlow: Longman Scientific \&| Technical; New York: John Wiley \&|
  Sons, Inc., 1989.

\bibitem{trev_tvs}
Fran\c{c}ois Tr\`eves.
\newblock {\em Topological vector spaces, distributions and kernels}.
\newblock Dover Publications, Inc., Mineola, NY, 2006.
\newblock Unabridged republication of the 1967 original.

\bibitem{vdv_cvx}
M.~L.~J. van~de Vel.
\newblock {\em Theory of convex structures}, volume~50 of {\em North-Holland
  Math. Libr.}
\newblock Amsterdam: North-Holland, 1993.

\bibitem{MR450972}
Heinrich von Weizs\"acker.
\newblock Some negative results in the theory of lifting.
\newblock In {\em Measure theory ({P}roc. {C}onf., {O}berwolfach, 1975)},
  volume Vol. 541 of {\em Lecture Notes in Math.}, pages 159--172. Springer,
  Berlin-New York, 1976.

\end{thebibliography}
%\bibliographystyle{plain}

\def\polhk#1{\setbox0=\hbox{#1}{\ooalign{\hidewidth
  \lower1.5ex\hbox{`}\hidewidth\crcr\unhbox0}}}

\Addresses

\end{document}